\documentclass[reqno]{amsart}
\usepackage{amssymb,graphicx}%
\usepackage{ifpdf}
\ifpdf
 \usepackage[hyperindex]{hyperref}
\else
 \expandafter\ifx\csname dvipdfm\endcsname\relax
 \usepackage[hypertex,hyperindex]{hyperref}
 \else
 \usepackage[dvipdfm,hyperindex]{hyperref}
 \fi
\fi
\allowdisplaybreaks[4]
\theoremstyle{plain}
\newtheorem{thm}{Theorem}[section]
\newtheorem{lem}{Lemma}[section]
\newtheorem{cor}{Corollary}[section]
\theoremstyle{remark}

\numberwithin{equation}{section}
 \makeatletter

\begin{document}

\title[Gradient estimates and Liouville theorem]
{\textbf{H\MakeLowercase{amilton}-S\MakeLowercase{ouplet}-Z\MakeLowercase{hang type Gradient estimates for Porous Medium type equations on} R\MakeLowercase{iemannian manifolds}}}

\author[W. Wang]{Wen Wang$^{1,2}$\quad Hui Zhou$^{1,2}$}
\address[W. Wang]{1. School of Mathematics and Statistics, Hefei Normal
University, Hefei 230601,P.R.China;}
\address{2. School of mathematical Science, University of Science and
Technology of China, Hefei 230026, China}
\email{\href{mailto: W. Wang <wwen2014@mail.ustc.edu.cn>}{wwen2014@mail.ustc.edu.cn}}

\begin{abstract}
In this paper, by employ the cutoff function and the maximum principle, some Hamilton-Souplet-Zhang type  gradient estimates for porous medium type equation
are deduced. As a special case, an Hamilton-Souplet-Zhang type  gradient estimates of the heat equation is derived which is different from the result of Souplet-Zhang.
 Furthermore, our results generalize those of Zhu. As application,  some Livillous theorems for ancient solution are derived.
\end{abstract}



\thanks{1.School of Mathematics and Statistics, Hefei Normal University, Hefei 230601, P. R. China }

  \thanks{2.School of mathematical Science, University of Science and
Technology of China, Hefei 230026, China }

\thanks{Corresponding author: Wen Wang, E-mail: wwen2014@mail.ustc.edu.cn}

\thanks{This work was supported by the Universities Natural Science Foundation of Anhui
Province (KJ2016A310);2017 Anhui Province outstanding young talent support project(gxyq2017048)}

\maketitle

\section{\textbf{Introduction and Main results}}
In the paper, let $(M^{n},g)$ be an $n$-dimensional complete Riemannian manifold. We consider the porous medium type equations
\begin{equation}\label{1.1}
u_{t}=\Delta u^{m}+\lambda(x,t) u^{l}, m>1
\end{equation}
on $(M^{n},g)$, where $l$ and $m$ are two real numbers, and $\lambda(x,t)\geq 0$ is defined on
$M^{n}\times[0,\infty)$ which is $C^2$ in the first variable and $C^1$
in the second variable.

The famous porous medium  equations (PME for short)
\begin{equation}\label{1.2}
u_{t}=\Delta u^{m}, m>1
\end{equation}
are of great interest because of important in mathematics, physics, and applications in many other fields. For $m=1$ it is the famous heat equation. As $m>1$, it is called the porous medium equation, and it has
arisen in different applications to model diffusive phenomena, such as, groundwater infiltration (Boussinesq's model, 1903, with $m=2$), flow of gas
in porous media (Leibenzon-Muskat model, $m\geq 2$), heat radiation in plasmas ($m>4$), liquid thin films moving under gravity ($m=4$), and others.
We can read a work by C\'{a}zqucz \cite{[19]} for basic theory and various applications of the porous medium equation in the Euclidean space. In the case $m<1$, it is said to be the fast diffusion equation.

In 1979, Aronson and B\'{e}nilan \cite{[1]} obtained a famous second order differential inequality
\begin{equation}\label{1.3}
\sum_{i}\frac{\partial}{\partial x_i}\left(mu^{m-2}\frac{\partial u}{\partial x_i}\right)\geq -\frac{\kappa}{t},\quad \kappa=\frac{n}{n(m-1)+2},
\end{equation}
for all positive solutions of ~\eqref{1.2} on the Euclidean space $\mathbb{R}^{n}$ with $m>1-\frac{2}{n}$.

Generalized research on PME ~\eqref{1.2} also attracted  many researchers'  interest.
In 1993, Hui \cite{[7]} considered the asymptotic  behaviour for solutions to equation
\begin{equation}\label{1.4}
u_{t}=\Delta u^{m}- u^{p}
\end{equation}
as $l\rightarrow\infty$. In 1994, Zhao and Yuan \cite{[23]}  proved the uniqueness of the solutions to equation ~\eqref{1.4} with
initial datum a measure. In 1997, Kawanago \cite{[11]} demonstrated existence and behaviour for solutions to equation
\begin{equation}\label{1.5}
u_{t}=\Delta u^{m}+ u^{l}.
\end{equation}
In 2001, E. Chasseigne \cite{[4]} investigated the initial trace for the equation ~\eqref{1.4}
in a cylinder $\Omega\times[0,T]$, where $\Omega$ is a regular, bounded open subset of $\mathbb{R}^{n}$ and $T>0$, $m>1$, and $q$ are constants.
Recently, Xie, Zheng and Zhou \cite{[21]} studied global existence for equation
\begin{equation}\label{1.6}
u_{t}=\Delta u^{m}- u^{p(x)}
\end{equation}
in $\Omega\times(0,T)$, where $p(x)>0$ is continuous function satisfying $0<p_{-}=\inf p(x)\leq p(x)\leq p_{+}=\sup p(x)<\infty$.

Recently, regularity estimates of PME ~\eqref{1.2} on manifolds are investigated.
In 2009, Lu, Ni, V\'{a}zquez and Villani \cite{[14]} studied the PME on an $n$-dimensional complete manifold $(M^{n},g)$, they
obtained a local Aronson-B\'{e}nilan estimate. Huang, Huang and Li in \cite{[8]} improved the part results of Lu,  Ni, V\'{a}zquez and Villani. In this article, we will
study Hamilton-Souplet-Zhang type gradient estimates to equation ~\eqref{1.1}.

Let First recall some known results.

\textbf{Theorem A (Hamilton \cite{[6]})}.   \emph{Let $(\mathbf{M}^{n}, g)$ be a closed
Riemannian manifold with  $Ricci(\mathbf{M})\geq -k$
for some $k\geq 0$. Suppose that $u$ is arbitrary positive solution
to the heat equation
\begin{equation}\label{1.7}
u_{t}=\Delta u
\end{equation}
and $u\leq M$. Then
\begin{equation}\label{1.8}
\frac{|\nabla u^{2}(x,t)|}{u^{2}(x,t)}\leq
C\left(\frac{1}{t}+2k\right)\log\frac{M}{u(x,t)}.
\end{equation}
}

In 2006, Souplet and Zhang \cite{[18]} generalized Hamilton's result, and obtained the corresponding gradient estimate and Liouville theorem.

\textbf{Theorem B (Souplet-Zhang [18])}.   \emph{Let $(\mathbf{M}^{n}, g)$ be a
Riemannian manifolds with $n\geq 2$ and $Ricci(\mathbf{M})\geq -k$
for some $k\geq 0$. Suppose that $u$ is arbitrary positive solution
to the heat equation ~\eqref{1.7} in $Q_{R,T}\equiv B(x_{0},R)\times [t_{0}-T,
t_{0}]\subset \mathbf{M}^{n}\times (-\infty, \infty)$ and $u\leq M$ in
$Q_{R,T}$. Then
\begin{equation}\label{1.9}
\frac{|\nabla u(x,t)|}{u(x,t)}\leq
C\left(\frac{1}{R}+\frac{1}{\sqrt{T}}+\sqrt{k}\right)\left(1+\log\frac{M}{u(x,t)}\right)
\end{equation}
in $Q_{\frac{R}{2},\frac{T}{2}}$, where $C$ is a dimensional
constant.}

 In 2013, Zhu \cite{[26]} deduced a Hamilton's gradient estimate and Liouville theorem for PME ~\eqref{1.2} on noncompact Riemannian manifolds. Huang, Xu and Zeng in \cite{[9]} improve the result of Zhu.

 \textbf{Theorem C (Zhu \cite{[26]})}.   \emph{Let $(\mathbf{M}^{n}, g)$ be a
Riemannian manifolds with $n\geq 2$ and $Ricci(\mathbf{M})\geq -k$
for some $k\geq 0$. Suppose that $u$ is arbitrary positive solution
to the PME ~\eqref{1.2} in $Q_{R,T}\equiv B(x_{0},R)\times [t_{0}-T,
t_{0}]\subset \mathbf{M}^{n}\times (-\infty, \infty)$. Let $v=\frac{m}{m-1}u^{m-1}$ and $v\leq M$. Then for $1<m<1+\frac{1}{\sqrt{2n}+1}$
\begin{equation}\label{1.10}
\frac{|\nabla v|}{v^{\frac{m-2}{4(m-1)}}}\leq
CM^{1+\frac{2-m}{4(m-1)}}\left(\frac{1}{R}+\frac{1}{\sqrt{T}}+\sqrt{k}\right).
\end{equation}
}

Recently, Cao and Zhu \cite{[3]} obtained some Aronson and B\'{e}nilan estimates for PME ~\eqref{1.2} under Ricci flow.

Our results of this paper are encouraged by the  work in Ref. \cite{[10],[12],[14],[15],[16],[17],[18],[28],[21],[26]}.
 We consider  the  porous medium type equation ~\eqref{1.1}, and deduce some Hamilton-Souplet-Zhang type  gradient estimates.

Our main results state as follows.
\begin{thm}
 Let $(M^{n}, g)$ be a
Riemannian manifold with dimensional $n$. Suppose that  $Ric(M^{n})\geq -k$ with $k\geq 0$.
If $u(x,t)$ is a positive solution of the  equation ~\eqref{1.1}
in $Q_{R,T}:=B_{x_0}(R)\times [t_{0}-T, t_0]\subset M^{n}\times(-\infty,\infty)$. Let $v=\frac{m}{m-1}u^{m-1}$ and $v\leq M$.
Also suppose that there exist two positive numbers $\delta$ and $\epsilon$ such that
$\lambda(x,t)\leq \delta$ and $|\nabla\lambda|^{2}\leq \epsilon|\lambda|$.
Then for  $1<m<1+\sqrt{\frac{1}{n-1}}$ and $l\geq 1-m$,
\begin{align}\label{1.11}
\nonumber
\frac{|\nabla v|}{v^{\frac{\beta}{2}}}(x,t)\leq& C\gamma^{2}(m-1)M^{1-\frac{\beta}{2}}\left(\frac{1}{R}+\sqrt{k}+\frac{1}{\sqrt{T}}\right)\\
&+C_{3}\big(\delta^{\frac{1}{2}} M^{\frac{m+l-1}{2(m-1)}}+\epsilon^{\frac{1}{4}} M^{\frac{3m+l-2}{4(m-1)}}\big)
\end{align}
in $Q_{\frac{R}{2},\frac{T}{2}}$, where $\beta=-\frac{1}{m-1}$, $\gamma=\frac{8}{1-(m-1)^{2}(n-1)}$, $C_{3}=C_{3}(m,n,l)$ and $C$ is a constant.
\end{thm}

When $\lambda(x,t)=0$, we get the following:
\begin{cor} Let $(M^{n}, g)$ be a
Riemannian manifold with dimensional $n$. Suppose that  $Ric(M^{n})\geq -k$ with $k\geq 0$.
If $u(x,t)$ is a positive solution of the PME ~\eqref{1.2}
in $Q_{R,T}:=B_{x_0}(R)\times [t_{0}-T, t_0]\subset M^{n}\times(-\infty,\infty)$. Let $v=\frac{m}{m-1}u^{m-1}$ and $v\leq M$. Then for  $1<m<1+\sqrt{\frac{1}{n-1}}$

\begin{align}\label{1.12}
\frac{|\nabla v|}{v^{\frac{\beta}{2}}}(x,t)\leq& C\gamma^{2}(m-1)M^{1-\frac{\beta}{2}}\left(\frac{1}{R}+\sqrt{k}+\frac{1}{\sqrt{T}}\right)
\end{align}
in $Q_{\frac{R}{2},\frac{T}{2}}$, where $\beta=-\frac{1}{m-1}$, $\gamma=\frac{8}{1-(m-1)^{2}(n-1)}$  and $C$ is a constant.
\end{cor}

Take $\lambda(x,t)=0$ and $m\searrow 1$ in Corollary $1.1$, the following estimate is derived.
\begin{cor} Let $(M^{n}, g)$ be a
Riemannian manifold of dimensional $n$. Suppose that  $Ric(M^{n})\geq -k$ with $k\geq 0$.
If $u(x,t)$ is a positive solution of the heat equation
\begin{equation*}
u_{t}=\Delta u,
\end{equation*}
in $Q_{R,T}:=B_{x_0}(R)\times [t_{0}-T, t_0]\subset M^{n}\times(-\infty,\infty)$.
Then  we have for $u\leq M$
\begin{align}\label{1.13}
\frac{|\nabla u|}{\sqrt{u}}(x,t)\leq C\left(\frac{1}{R}+\sqrt{k}+\frac{1}{\sqrt{T}}\right)
\end{align}
in $Q_{\frac{R}{2},\frac{T}{2}}$, where  $C$ is a constant.
\end{cor}

\begin{thm}
 Let $(M^{n}, g)$ be a
Riemannian manifold with dimensional $n$. Suppose that  $Ric(M^{n})\geq -k$ with $k\geq 0$.
If $u(x,t)$ is a positive solution of the  equation ~\eqref{1.1}
in $Q_{R,T}:=B_{x_0}(R)\times [t_{0}-T, t_0]\subset M^{n}\times(-\infty,\infty)$. Let $v=\frac{m}{m-1}u^{m-1}$ and $v\leq M$.
Also suppose that there exist a positive number $\epsilon$ such that
 $|\nabla\lambda|^{2}\leq \epsilon|\lambda|$.
Then for  $1<m<1+\sqrt{\frac{1}{n-1}}$ and $2-3m\leq l\leq 2-\frac{3}{2}m$,
\begin{align}\label{1.14}
\frac{|\nabla v|}{v^{\frac{\beta}{2}}}(x,t)\leq C\gamma^{2}(m-1)M^{1-\frac{\beta}{2}}\left(\frac{1}{R}+\sqrt{k}+\frac{1}{\sqrt{T}}\right)+C_{3}\epsilon^{\frac{1}{4}} M^{\frac{3m+l-2}{4(m-1)}}
\end{align}
in $Q_{\frac{R}{2},\frac{T}{2}}$, where $\beta=-\frac{1}{m-1}$, $\gamma=\frac{8}{1-(m-1)^{2}(n-1)}$, $C_{3}=C_{3}(m,n,l)$ and $C$ is a constant.
\end{thm}

\textbf{Remark:} (a)~  Since $1+\sqrt{\frac{1}{n-1}}>1+\sqrt{\frac{1}{2n+1}}$, so the result of Corollary $1.1$ in the paper generalize those of Zhu in \cite{[26]}.

(b)~  ~When $\lambda(x,t)=0$, the result of Corollary $1.1$ in the paper is the result of Huang, Xie and Zeng in \cite{[9]}.

(c)~  ~\eqref{1.13} is different from Souplet-Zhang's result in \cite{[18]}. Moreover,  our results in form  seem to be simpler than Souplet-Zhang's result.

\section{\textbf{Preliminary}}
In this section, we derive a lemma.

\begin{lem}\cite{[21]}\quad
Let $A=(a_{ij})$ be a nonzero $n\times n$ symmetric matrix with eigenvalues $\lambda_k$, for any
$a,b\in \mathbf{R}$, then
$$\max_{A\in S(n);|v|=1}\left[\frac{aA+b\mathrm{tr} AI_{n}}{|A|}(v,v)\right]^{2}=(a+b)^{2}+(n-1)b^{2}.$$
\end{lem}

\begin{lem}\quad
Let $1<m<1+\sqrt{\frac{1}{n-1}}$ and $\theta=\frac{1-(m-1)^{2}(n-1)}{4(m-1)}>0$. Then we have
\begin{align}\label{2.1}
\nonumber
(m-1)v\Delta w-w_{t}
\geq & \theta w^{2}v^{\beta-1}-2(m-1)kwv-m \nabla w\cdot\nabla v\\
\nonumber
&+\lambda\Big[\beta(m-1)-2(m+l-2)\Big]\left(\frac{m-1}{m}v\right)^{\frac{l-1}{m-1}}w\\
&-(m-1)\Big(\frac{m-1}{m}v\Big)^{\frac{l-1}{m-1}}\left(|\lambda|w+\frac{|\nabla\lambda|^2}{|\lambda|}\cdot\frac{1}{v^{\beta-2}}\right).
\end{align}
\end{lem}

\begin{proof}
Let $v=\frac{m}{m-1}u^{m-1}$, then
\begin{equation}\label{2.2}
v_{t}=(m-1)v\Delta v+|\nabla v|^{2}+\lambda (m-1)\left(\frac{m-1}{m}\right)^{\frac{l-1}{m-1}}v^{1+\frac{l-1}{m-1}}.
\end{equation}
Let $w=\frac{|\nabla v|^2}{v^\beta}$, then
\begin{eqnarray}\label{2.3}
\nonumber
w_{t}&=&\frac{2v_{i}v_{it}}{v^\beta}-\beta\frac{v^{2}_{i}v_t}{v^{\beta+1}}\\
\nonumber
&=&\frac{2v_{i}\left[(m-1)v\Delta v+|\nabla v|^{2}+\lambda (m-1)\left(\frac{m-1}{m}\right)^{\frac{l-1}{m-1}}v^{1+\frac{l-1}{m-1}}\right]_i}{v^\beta}\\
\nonumber
&&-\beta\frac{v^{2}_{i}\left[(m-1)v\Delta v+|\nabla v|^{2}+\lambda (m-1)\left(\frac{m-1}{m}\right)^{\frac{l-1}{m-1}}v^{1+\frac{l-1}{m-1}}\right]}{v^{\beta+1}}\\
\nonumber
&=&2(m-1)\frac{v^{2}_{i}v_{jj}}{v^\beta}+2(m-1)\frac{v_{i}v_{jji}}{v^{\beta-1}}+4\frac{v_{i}v_{ij}v_{j}}{v^{\beta}}\\
\nonumber
&&+2\lambda\frac{(m+l-2)(\frac{m-1}{m}v)^{\frac{l-1}{m-1}}v^{2}_i}{v^\beta}+2(m-1)\Big(\frac{m-1}{m}v\Big)^{\frac{l-1}{m-1}}\frac{\nabla v\cdot\nabla\lambda}{v^{\beta-1}}\\
&&-\beta(m-1)\frac{v^{2}_{i}v_{jj}}{v^\beta}-\beta\frac{v^{2}_{i}v^{2}_{j}}{v^{\beta+1}}-\lambda\beta(m-1)\left(\frac{m-1}{m}v\right)^{\frac{l-1}{m-1}}\frac{v^{2}_i}{v^\beta},
\end{eqnarray}
\begin{equation}\label{2.4}
w_{j}=\frac{2v_{i}v_{ij}}{v^\beta}-\beta\frac{v^{2}_{i}v_j}{v^{\beta}},
\end{equation}
\begin{equation}\label{2.5}
w_{jj}=\frac{2v^{2}_{ij}}{v^\beta}+\frac{2v_{i}v_{ijj}}{v^\beta}-4\beta\frac{v_{i}v_{ij}v_j}{v^{\beta+1}}-\beta\frac{v^{2}_{i}v_{jj}}{v^{\beta+1}}+\beta(\beta+1)\frac{v^{2}_{i}v^{2}_{j}}{v^{\beta+2}}.
\end{equation}
By ~\eqref{2.4} and ~\eqref{2.5}
\begin{align}\label{2.6}
\nonumber
&(m-1)v\Delta w-w_{t}\\
\nonumber
=&2(m-1)\frac{v^{2}_{ij}}{v^{\beta-1}}+2(m-1)\frac{v_{i}v_{ijj}}{v^{\beta-1}}-2(m-1)\frac{v_{i}v_{jji}}{v^{\beta-1}}-4\beta(m-1)\frac{v_{i}v_{ij}v_{j}}{v^{\beta}}\\
\nonumber
&+\beta(\beta+1)(m-1)\frac{v^{2}_{i}v^{2}_{j}}{v^{\beta+1}}-2(m-1)\frac{v^{2}_{i}v_{jj}}{v^{\beta}}-4\frac{v_{i}v_{ij}v_j}{v^\beta}+\beta\frac{v^{2}_{i}v^{2}_j}{v^{\beta+1}}\\
\nonumber
&-2\lambda(m+l-2)\left(\frac{m-1}{m}v\right)^{\frac{l-1}{m-1}}\frac{v^{2}_{i}}{v^\beta}-2(m-1)\Big(\frac{m-1}{m}v\Big)^{\frac{l-1}{m-1}}\frac{\nabla v\cdot\nabla\lambda}{v^{\beta-1}}\\
\nonumber
&+\lambda\beta(m-1)\left(\frac{m-1}{m}v\right)^{\frac{l-1}{m-1}}\frac{v^{2}_{i}}{v^\beta}\\
\nonumber
=&2(m-1)\frac{v^{2}_{ij}}{v^{\beta-1}}+2(m-1)\frac{R_{ij}v_{i}v_{j}}{v^{\beta-1}}-4[1+\beta(m-1)]\frac{v_{i}v_{ij}v_j}{v^\beta}-2(m-1)\frac{v^{2}_{i}v_{jj}}{v^\beta}\\
\nonumber
&+\beta\big[(\beta+1)(m-1)+1\big]\frac{v^{i}v^{2}_{j}}{v^{\beta+1}}-2\lambda(m+l-2)\left(\frac{m-1}{m}v\right)^{\frac{l-1}{m-1}}\frac{v^{2}_{i}}{v^\beta}\\
&-2(m-1)\Big(\frac{m-1}{m}v\Big)^{\frac{l-1}{m-1}}\frac{\nabla v\cdot\nabla\lambda}{v^{\beta-1}}
+\lambda\beta(m-1)\left(\frac{m-1}{m}v\right)^{\frac{l-1}{m-1}}\frac{v^{2}_{i}}{v^\beta},
\end{align}
Since
\begin{equation}\label{2.7}
\nabla w\cdot\nabla v=\frac{2v_{i}v_{ij}v_j}{v^\beta}-\beta\frac{v^{2}_{i}v^{2}_j}{v^{\beta+1}}.
\end{equation}
Adding $\varepsilon\times$ ~\eqref{2.7} to ~\eqref{2.6},
\begin{align*}
\nonumber
&(m-1)v\Delta w-w_{t}\\
\nonumber
=&2(m-1)\frac{v^{2}_{ij}}{v^{\beta-1}}+2(m-1)\frac{R_{ij}v_{i}v_{j}}{v^{\beta-1}}+\big[2\varepsilon-4(1+\beta(m-1)\big]\frac{v_{i}v_{ij}v_j}{v^\beta}\\
\nonumber
&-2(m-1)\frac{v^{2}_{i}v_{jj}}{v^\beta}+\beta\big[(\beta+1)(m-1)+1-\varepsilon\big]\frac{v^{i}v^{2}_{j}}{v^{\beta+1}}-\varepsilon \nabla w\cdot\nabla v\\
&-2\lambda(m+l-2)\left(\frac{m-1}{m}v\right)^{\frac{l-1}{m-1}}\frac{v^{2}_{i}}{v^\beta}-2(m-1)\Big(\frac{m-1}{m}v\Big)^{\frac{l-1}{m-1}}\frac{\nabla v\cdot\nabla\lambda}{v^{\beta-1}}\\
\nonumber
&+\lambda\beta(m-1)\left(\frac{m-1}{m}v\right)^{\frac{l-1}{m-1}}\frac{v^{2}_{i}}{v^\beta}\\
\nonumber
=&2(m-1)\frac{|A|^2}{v^{\beta-1}}+2(m-1)R_{ij}wv+\big[2\varepsilon-4(1+\beta(m-1)\big]\frac{A(e,e)}{|A|}w|A|\\
\nonumber
&-2(m-1)\frac{\mathrm{tr} A}{|A|}w|A|+\beta\big[(\beta+1)(m-1)+1-\varepsilon\big]w^{2}v^{\beta-1}-\varepsilon \nabla w\cdot\nabla v\\
\nonumber
&-2\lambda(m+l-2)\left(\frac{m-1}{m}v\right)^{\frac{l-1}{m-1}}w+\lambda\beta(m-1)\left(\frac{m-1}{m}v\right)^{\frac{l-1}{m-1}}w\\
\nonumber
&-2(m-1)\Big(\frac{m-1}{m}v\Big)^{\frac{l-1}{m-1}}\frac{\nabla v\cdot\nabla\lambda}{v^{\beta-1}}\\
\nonumber
=&2(m-1)\frac{|A|^2}{v^{\beta-1}}+\left\{\Big[2\varepsilon-4[1+\beta(m-1)]\Big]\frac{A(e,e)}{|A|}-2(m-1)\frac{\mathrm{tr}A}{|A|}\right\}w|A|\\
\nonumber
&+2(m-1)R_{ij}wv+\beta\big[(\beta+1)(m-1)+1-\varepsilon\big]w^{2}v^{\beta-1}-\varepsilon \nabla w\cdot\nabla v\\
\nonumber
&+\lambda\Big[\beta(m-1)-2(m+l-2)\Big]\left(\frac{m-1}{m}v\right)^{\frac{l-1}{m-1}}w-2(m-1)\Big(\frac{m-1}{m}v\Big)^{\frac{l-1}{m-1}}\frac{\nabla v\cdot\nabla\lambda}{v^{\beta-1}}\\
\nonumber
=&2(m-1)\left\{\frac{|A|}{v^{\frac{\beta-1}{2}}}+\frac{1}{4(m-1)}\left[\Big(2\varepsilon-4[1+\beta(m-1)]\Big)\frac{A(e,e)}{|A|}-2(m-1)\frac{\mathrm{tr}A}{|A|}\right]wv^{\frac{\beta-1}{2}}\right\}^{2}\\
\nonumber
&-\frac{1}{8(m-1)}\left[\Big(2\varepsilon-4[1+\beta(n-1)]\Big)\frac{A(e,e)}{|A|}-2(m-1)\frac{\mathrm{tr}A}{|A|}\right]^{2}w^{2}v^{\beta-1}\\
\nonumber
&+2(m-1)R_{ij}wv+\beta\big[(\beta+1)(m-1)+1-\varepsilon\big]w^{2}v^{\beta-1}-\varepsilon \nabla w\cdot\nabla v\\
\nonumber
&+\lambda\Big[\beta(m-1)-2(m+l-2)\Big]\left(\frac{m-1}{m}v\right)^{\frac{l-1}{m-1}}w-2(m-1)\Big(\frac{m-1}{m}v\Big)^{\frac{l-1}{m-1}}\frac{\nabla v\cdot\nabla\lambda}{v^{\beta-1}}\\
\nonumber
\geq& -\frac{1}{8(m-1)}\left[\Big(2\varepsilon-4[1+\beta(m-1)]\Big)\frac{A(e,e)}{|A|}-2(m-1)\frac{\mathrm{tr}A}{|A|}\right]^{2}w^{2}v^{\beta-1}\\
\nonumber
&+2(m-1)R_{ij}wv+\beta\big[(\beta+1)(m-1)+1-\varepsilon\big]w^{2}v^{\beta-1}-\varepsilon \nabla w\cdot\nabla v\\
\nonumber
&+\lambda\Big[\beta(m-1)-2(m+l-2)\Big]\left(\frac{m-1}{m}v\right)^{\frac{l-1}{m-1}}w-2(m-1)\Big(\frac{m-1}{m}v\Big)^{\frac{l-1}{m-1}}\frac{\nabla v\cdot\nabla\lambda}{v^{\beta-1}},
\end{align*}
where $A_{ij}=(v_{ij})$ and  $e=\nabla v/|\nabla v|$. By  applying Lemma $2.1$ with
$a=2\varepsilon-4[1+\beta(n-1)$ and $b=-2(m-1)$,
\begin{align}\label{2.8}
\nonumber
&(m-1)v\Delta w-w_{t}\\
\nonumber
\geq &-\frac{1}{8(m-1)}\left\{\Big[2\varepsilon-4\big[1+\beta(m-1)\big]-2(m-1)\Big]^{2}+4(m-1)^{2}(n-1)\right\}w^{2}v^{\beta-1}\\
\nonumber
&+2(m-1)R_{ij}wv+\beta\big[(\beta+1)(m-1)+1-\varepsilon\big]w^{2}v^{\beta-1}-\varepsilon \nabla w\cdot\nabla v\\
\nonumber
&+\lambda\Big[\beta(m-1)-2(m+l-2)\Big]\left(\frac{m-1}{m}v\right)^{\frac{l-1}{m-1}}w-2(m-1)\Big(\frac{m-1}{m}v\Big)^{\frac{l-1}{m-1}}\frac{\nabla v\cdot\nabla\lambda}{v^{\beta-1}}\\
\nonumber
=&-\frac{1}{8(m-1)}f(\beta,\varepsilon)w^{2}v^{\beta-1}+2(m-1)R_{ij}wv-\varepsilon \nabla w\cdot\nabla v\\
\nonumber
&+\lambda\Big[\beta(m-1)-2(m+l-2)\Big]\left(\frac{m-1}{m}v\right)^{\frac{l-1}{m-1}}w\\
&-2(m-1)\Big(\frac{m-1}{m}v\Big)^{\frac{l-1}{m-1}}\frac{\nabla v\cdot\nabla\lambda}{v^{\beta-1}},
\end{align}
where
\begin{align}\label{2.9}
\nonumber
f(\beta,\varepsilon)=&\Big[2\varepsilon-4\big[1+\beta(m-1)\big]-2(m-1)\Big]^{2}+4(m-1)^{2}(n-1)\\
&-8(m-1)\beta\big[(\beta+1)(m-1)+1-\varepsilon\big].
\end{align}
For the purpose of showing that the coefficient of $w^{2}v^{\beta-1}$ is positive, we minimize the function $f(\beta,\varepsilon)$ by letting
$\varepsilon=m$ and $\beta=-\frac{1}{m-1}$,
such that
$$f(\beta,\varepsilon)=4(m-1)^{2}(n-1)-4.$$
Then ~\eqref{2.8} becomes
\begin{align}\label{2.10}
\nonumber
(m-1)v\Delta w-w_{t}
\geq &\frac{1-(m-1)^{2}(n-1)}{4(m-1)}w^{2}v^{\beta-1}-2(m-1)kwv-m \nabla w\cdot\nabla v\\
\nonumber
&+\lambda\Big[\beta(m-1)-2(m+l-2)\Big]\left(\frac{m-1}{m}v\right)^{\frac{l-1}{m-1}}w\\
\nonumber
&-2(m-1)\Big(\frac{m-1}{m}v\Big)^{\frac{l-1}{m-1}}\frac{|\nabla v|\cdot|\nabla\lambda|}{v^{\beta-1}}\\
\nonumber
=& \theta w^{2}v^{\beta-1}-2(m-1)kwv-m \nabla w\cdot\nabla v\\
\nonumber
&+\lambda\Big[\beta(m-1)-2(m+l-2)\Big]\left(\frac{m-1}{m}v\right)^{\frac{l-1}{m-1}}w\\
&-(m-1)\Big(\frac{m-1}{m}v\Big)^{\frac{l-1}{m-1}}\left(|\lambda|w+\frac{|\nabla\lambda|^2}{|\lambda|}\cdot\frac{1}{v^{\beta-2}}\right),
\end{align}
where $\theta=\frac{1-(m-1)^{2}(n-1)}{4(m-1)}>0$ as $1<m<1+\sqrt{\frac{1}{n-1}}$, and in the last inequality we utilize the fact that
\begin{align*}
&-2(m-1)\Big(\frac{m-1}{m}v\Big)^{\frac{l-1}{m-1}}\frac{|\nabla v|\cdot|\nabla\lambda|}{v^{\beta-1}}\\
&\geq-(m-1)\Big(\frac{m-1}{m}v\Big)^{\frac{l-1}{m-1}}\left(|\lambda|w+\frac{|\nabla\lambda|^2}{|\lambda|}\cdot\frac{1}{v^{\beta-2}}\right).
\end{align*}
\end{proof}

\section{\textbf{Proof of main results}}

From here, we will utilize the well-known cut-off function of Li and
Yau to derive the desire bounds.

\begin{proof}[\textbf{Proof of Theorem 1.1}]\quad Assume that a function $\Psi=\Psi(x,t)$  is a smooth cut-off
function
supported in $Q_{R,T}$, satisfying the following properties,\\
(1) $\Psi=\Psi(d(x,x_{0}),t)\equiv \psi(r,t)$; $\Psi(r,t)=1$ in
$Q_{R/2,T/2}$, $0\leq \Psi\leq 1$.\\
(2) $\Psi$ is decreasing as a radial function in the spatial
variables.\\
(3) $|\partial_{r}\Psi|/\Psi^{a}\leq C_{a}/R$,
$|\partial^{2}_{r}\Psi|/\Psi^{a}\leq C_{a}/R^{2}$ when $0<a<1$.\\
(4) $|\partial_{t}\Psi|/\Psi^{1/2}\leq C/T$.

Assume that the maximum of $\Psi w$ is arrived at point $(x_{1},
t_{1})$. By \cite{[13]}, we can suppose, without loss of generality, that
$x_1$ is not on the cut-locus of $\mathbf{M}^n$. Therefore, at
$(x_{1}, t_1)$, it yields $\Delta(\Psi w)\leq 0$, $(\Psi w)_{t}\geq
0$ and $\nabla(\Psi w)=0$.  Hence, by ~\eqref{2.1} and a straightforward
calculation, it yields that
\begin{align}\label{3.1}
\nonumber
0\geq& \Big[(m-1)v\Delta-\partial_{t}\Big](\Psi w)\\
\nonumber
=& \Psi\big[(m-1)v\Delta-\partial_{t}\big]w+(m-1)vw\Delta\Psi-w\Psi_{t}+2(m-1)\frac{v}{\Psi}\nabla\Psi\cdot\nabla(\Psi w)\\
\nonumber
&-2(m-1)vw\frac{|\nabla\Psi|^2}{\Psi}\\
\nonumber
=&\Psi\theta w^{2}v^{\beta-1}-2(m-1)\Psi kwv-p \nabla (\Psi w)\cdot\nabla v+mw\nabla v\cdot\nabla\Psi\\
\nonumber
&+\Psi\lambda\Big[\beta(m-1)-2(m+l-2)\Big]\left(\frac{m-1}{m}v\right)^{\frac{l-1}{m-1}}w\\
\nonumber
&-(m-1)\Psi\Big(\frac{m-1}{m}v\Big)^{\frac{l-1}{m-1}}\left(\lambda w+\frac{|\nabla\lambda|^2}{\lambda}\cdot\frac{1}{v^{\beta-2}}\right)+(m-1)vw\Delta\Psi\\
&-w\Psi_{t}+2(m-1)\frac{v}{\Psi}\nabla\Psi\cdot\nabla(\Psi w)-2(m-1)vw\frac{|\nabla\Psi|^2}{\Psi}.
\end{align}
Then ~\eqref{3.1} becomes at the point $(x_{1},t_{1})$
\begin{align}\label{3.2}
\nonumber
\Psi\theta w^{2}v^{\beta-1}\leq &2(m-1)\Psi kwv-mw\nabla v\cdot\nabla\Psi-(m-1)vw\Delta\Psi+2(m-1)vw\frac{|\nabla\Psi|^2}{\Psi}\\
\nonumber
&+w\Psi_{t}-\Psi\lambda\Big[\beta(m-1)-2(m+l-2)\Big]\left(\frac{m-1}{m}v\right)^{\frac{l-1}{m-1}}w\\
&+(m-1)\Psi\Big(\frac{m-1}{m}v\Big)^{\frac{l-1}{m-1}}\left(\lambda w+\frac{|\nabla\lambda|^2}{\lambda}\cdot\frac{1}{v^{\beta-2}}\right).
\end{align}
Now setting $\theta=\frac{2}{\gamma}\cdot\frac{1}{m-1}$ and $\gamma=\frac{8}{1-(m-1)^{2}(n-1)}$, then ~\eqref{3.2} gives
\begin{align}\label{3.3}
\nonumber
2\Psi w^{2}\leq &2\gamma(m-1)^{2}\Psi kwv^{2-\beta}-\gamma(m-1) v^{1-\beta}mw\nabla v \cdot\nabla\Psi-(m-1)^{2}\gamma v^{2-\beta}w\Delta\Psi\\
\nonumber
&+2(m-1)^{2}\gamma v^{2-\beta}w\frac{|\nabla\Psi|^2}{\Psi}+\gamma(m-1) w\Psi_{t}v^{1-\beta}\\
\nonumber
&-(m-1)\gamma\Psi\lambda\Big[(\beta-1)(m-1)-2(m+l-2)\Big]\left(\frac{m-1}{m}v\right)^{\frac{l-1}{m-1}}w v^{1-\beta}\\
&+(m-1)^{2}\gamma\Psi\Big(\frac{m-1}{m}v\Big)^{\frac{l-1}{m-1}}\frac{|\nabla\lambda|^2}{\lambda}\cdot\frac{v^{1-\beta}}{v^{\beta-2}}.
\end{align}
Now, we need to search for an upper bound for each term on the right-hand side of ~\eqref{3.3}. After a siample calculation,
it is not diffucult to find the following estimates.
\begin{align}\label{3.4}
&2\gamma(m-1)^{2}\Psi kwv^{2-\beta}\leq \frac{1}{4}\Psi w^{2}+C\gamma^{2}(m-1)^{4}M^{4-2\beta}k^{2},
\end{align}
\begin{align}\label{3.5-3.8}
&-(m-1)\gamma v^{1-\beta}mw\nabla v \cdot\nabla\Psi\leq \frac{1}{4}\Psi w^{2}+C\gamma^{2}(m-1)^{4}M^{4-2\beta}\frac{1}{R^{4}},\\
&-(m-1)^{2}\gamma v^{2-\beta}w\Delta\Psi\leq\frac{1}{4}\Psi w^{2}+C\gamma^{2}(m-1)^{4}M^{4-2\beta}\left(\frac{1}{R^{4}}+\frac{k}{R^2}\right),\\
&2(m-1)^{2}\gamma v^{2-\beta}w\frac{|\nabla\Psi|^2}{\Psi}\leq \frac{1}{4}\Psi w^{2}+C\gamma^{2}(m-1)^{4}M^{4-2\beta}\frac{1}{R^{4}},\\
&(m-1)\gamma w\Psi_{t}v^{1-\beta}\leq\frac{1}{4}\Psi w^{2}+C\gamma^{2}(m-1)^{4}M^{4-2\beta}\frac{1}{T^{2}},
\end{align}
where  $C$ is a constant and we used the fact that $0<v\leq M$, $\beta=-\frac{1}{m-1}<0$.

Applying $0<v\leq M$, $\beta=-\frac{1}{m-1}<0$ and $m>1$  we now give estimates for the last two items of ~\eqref{3.3}.
\begin{align}\label{3.9}
\nonumber
&-(m-1)\gamma\Psi\lambda\Big[(\beta+1)(m-1)-2(m+l-2)\Big]\left(\frac{m-1}{m}v\right)^{\frac{l-1}{m-1}}w v^{1-\beta}\\
&\leq\frac{1}{4}\Psi w^{2}+C_{1} \delta^{2}M^{\frac{2m+2l-2}{m-1}},
\end{align}
where $C_{1}=C_{1}(m,n,l)$,  and inequality  holds for $l\leq 1-m$.
\begin{align}\label{3.10}
(m-1)^{2}\gamma\Psi\Big(\frac{m-1}{m}v\Big)^{\frac{l-1}{m-1}}\frac{|\nabla\lambda|^2}{|\lambda|}\cdot\frac{v^{1-\beta}}{v^{\beta-2}}\leq C_{2}\epsilon M^{\frac{3m+l-2}{m-1}}.
\end{align}
where $C_{2}=C_{2}(m,n,l)$, and  inequality is valid for $l\geq 2-3m$. Hence, both ~\eqref{3.9} and ~\eqref{3.10} hold for $l\geq 1-m$. \\
Substituting ~\eqref{3.4}--~\eqref{3.10} into ~\eqref{3.3}, we have for  $l\geq 1-m$ and $C_{3}=C_{3}(m,n,l)$
\begin{align}\label{3.11}
\nonumber
2\Psi w^{2}\leq& \frac{3}{2}\Psi w^{2}+C\gamma^{2}(m-1)^{4}M^{4-2\beta}\left(\frac{1}{R^{4}}+k^{2}+\frac{1}{T^2}\right)\\
&+C_{3}(\delta^{2} M^{\frac{2m+2l-2}{m-1}}+\epsilon M^{\frac{3m+l-2}{m-1}}),
\end{align}
which gives at the point $(x_{1},
t_{1})$
\begin{align}\label{3.12}
\nonumber
\Psi w^{2}\leq& C\gamma^{2}(m-1)^{4}M^{4-2\beta}\left(\frac{1}{R^{4}}+k^{2}+\frac{1}{T^2}\right)\\
&+C_{3}(\delta^{2} M^{\frac{2m+2l-2}{m-1}}+\epsilon M^{\frac{3m+l-2}{m-1}}).
\end{align}
Hnece, for all the point $(x,t)\in Q_{R,T}$,
\begin{align}\label{3.13}
\nonumber
(\Psi^{2} w^{2})(x,t)\leq&(\Psi^{2} w^{2})(x_{1},t_{1})\leq(\Psi w^{2})(x_{1},t_{1})\\
\nonumber
\leq& C\gamma^{2}(m-1)^{4}M^{4-2\beta}\left(\frac{1}{R^{4}}+k^{2}+\frac{1}{T^2}\right)\\
&+C_{3}(\delta^{2} M^{\frac{2m+2l-2}{m-1}}+\epsilon M^{\frac{3m+l-2}{m-1}}).
\end{align}
Notice that $\Psi=1$ in $Q_{R/2,T/2}$ and $w=\frac{|\nabla v|^2}{v^\beta}$. Therefore, we have for  $l\geq 1-m$,
\begin{align*}
\nonumber
\frac{|\nabla v|}{v^{\frac{\beta}{2}}}(x,t)\leq& C\gamma^{2}(m-1)M^{1-\frac{\beta}{2}}\left(\frac{1}{R}+\sqrt{k}+\frac{1}{\sqrt{T}}\right)\\
&+C_{3}\big(\delta^{\frac{1}{2}} M^{\frac{m+l-1}{2(m-1)}}+\epsilon^{\frac{1}{4}} M^{\frac{3m+l-2}{4(m-1)}}\big).
\end{align*}
\end{proof}

\begin{proof}[\textbf{Proof of Corollary 1.2}]\quad
By taking $\lambda(x,t)=0$ in ~\eqref{1.11}, we deduce that
\begin{align}\label{3.14}
\frac{|\nabla v|}{v^{\frac{\beta}{2}}}(x,t)\leq& C\gamma^{2}(m-1)M^{1-\frac{\beta}{2}}\left(\frac{1}{R}+\sqrt{k}+\frac{1}{\sqrt{T}}\right)
\end{align}
 Applying $v=\frac{m}{m-1}u^{m-1}$ to ~\eqref{3.14}, we obtain
\begin{align}\label{3.15}
m\cdot m^{\frac{1}{2(m-1)}}\frac{|\nabla u|}{u^{\frac{3}{2}-m}}(x,t)\leq& C\gamma^{2}\bigl[(m-1)M\bigr]^{1+\frac{1}{2(m-1)}}\left(\frac{1}{R}+\sqrt{k}+\frac{1}{\sqrt{T}}\right)
\end{align}
Since $(m-1)v=mu^{m-1}$, we have $(m-1)v\rightarrow 1$ as $m\searrow 1$. Therefore, we follow $(m-1)M\rightarrow 1$ as $m\searrow 1$. A sample computation yields
$$\lim_{m\rightarrow 1^+}\bigl[(m-1)M\bigr]^{\frac{1}{4(m-1)}}
=\lim_{m\rightarrow 1^+}\bigl[1+(m-1)M-1\bigr]^{\frac{1}{(m-1)M-1}\cdot\frac{(m-1)M-1}{2(m-1)}}=e^{\frac{1}{2}},$$
$$\lim_{m\rightarrow 1^+}\bigl[m\bigr]^{\frac{1}{2(m-1)}}=\lim_{m\rightarrow 1^+}\bigl[1+m-1\bigr]^{\frac{1}{m-1}\cdot\frac{1}{2}}=e^{\frac{1}{2}},$$
$$\lim_{m\rightarrow 1^+}\gamma=\lim_{m\rightarrow 1^+}\frac{8}{1-(m-1)^{2}(n-1)}=8.$$
Hence as  $m\searrow 1$, ~\eqref{3.15} becomes
\begin{align*}
\frac{|\nabla u|}{u^{\frac{1}{2}}}(x,t)\leq& C\left(\frac{1}{R}+\sqrt{k}+\frac{1}{\sqrt{T}}\right),
\end{align*}
where $C=C(n)$.
\end{proof}

\begin{proof}[\textbf{Proof of Theorem 1.2}]\quad
Since $\beta=-\frac{1}{m-1}<0$ and $m>1$, then $(\beta-1)(m-1)-2(m+l-2)\geq 0$ for $l\leq 2-\frac{3}{2}m$. Hence, ~\eqref{3.3} becomes
\begin{align}\label{3.16}
\nonumber
2\Psi w^{2}\leq &2\gamma(m-1)^{2}\Psi kwv^{2-\beta}-\gamma(m-1) v^{1-\beta}pw\nabla v \cdot\nabla\Psi-(m-1)^{2}\gamma v^{2-\beta}w\Delta\Psi\\
\nonumber
&+2(m-1)^{2}\gamma v^{2-\beta}w\frac{|\nabla\Psi|^2}{\Psi}+\gamma(m-1) w\Psi_{t}v^{1-\beta}\\
&+(m-1)^{2}\gamma\Psi\Big(\frac{m-1}{m}v\Big)^{\frac{l-1}{m-1}}\frac{|\nabla\lambda|^2}{\lambda}\cdot\frac{v^{1-\beta}}{v^{\beta-2}}.
\end{align}
A discussion of similar Theorem $1.1$ from  (3.4)-(3.8), ~\eqref{3.10} and ~\eqref{3.16},  we have for  $ 2-3m\leq l\leq 2-\frac{3}{2}m$ and $C_{3}=C_{3}(m,n,l)$
\begin{align*}
\frac{|\nabla v|}{v^{\frac{\beta}{2}}}(x,t)\leq C\gamma^{2}(m-1)M^{1-\frac{\beta}{2}}\left(\frac{1}{R}+\sqrt{k}+\frac{1}{\sqrt{T}}\right)+C_{3}\epsilon^{\frac{1}{4}} M^{\frac{3m+l-2}{4(m-1)}}.
\end{align*}
\end{proof}
\section{\textbf{Applications}}
In this section, we will deduce some related Liouville type theorems.

Applying Corollary $1.1$, it follows the following Liouville type theorem.
\begin{thm}[Liouville type theorem]
 Let $(M^{n}, g)$ be a complete, non-compact
Riemannian manifold with nonnegative Ricci curvature. Suppose that  $u$ is a positive ancient solution of the  equation ~\eqref{1.2}
such that $v(x,t)=o(d(x)+\sqrt{T})^{\frac{2(m-1)}{2m-1}}$, where $v=\frac{m}{m-1}u^{m-1}$. Then $u$ is a constant.
\end{thm}

By utilize Corollary $1.2$, the related Liouville type theorem is derived, as follows.
\begin{thm}[Liouville type theorem]
 Let $(M^{n}, g)$ be a complete, non-compact
Riemannian manifold with nonnegative Ricci curvature. Suppose that  $u$ is a positive ancient solution of the heat equation ~\eqref{1.7}
such that $u(x,t)=o(d(x)+\sqrt{T})^{2}$. Then $u$ is a constant.
\end{thm}

The proof of Theorem $4.1$ and Theorem $4.2$ are the same. So we only prove Theorem $4.1$.

\begin{proof}[Proof of Theorem 4.1]
Fix $(x_{0}, t_{0})$in space time. Assume that $u(x,t)$ is a positive ancient solution to PME ~\eqref{1.2} such that $v(x,t)=o(d(x)+\sqrt{T})^{\frac{2(m-1)}{2m-1}}$ near infity.
Applying ~\eqref{1.12} to $u$ on the cube $B(x_{0}, R)\times[t_{0}-R^{2}, t_0]$, then we have
\begin{align*}
v(x_{0}, t_0)^{\frac{1}{2(m-1)}}|\nabla v(x_{0}, t_0)|\leq \frac{C}{R}\cdot o(R).
\end{align*}
Let $R\rightarrow\infty$, we get $|\nabla v(x_{0}, t_0)|=0$. Therefore, the result are
derived.
\end{proof}

\section{\textbf{Acknowledgement}}
We are grateful to Professor Jiayu Li for his encouragement. We also thank Professor Qi S Zhang for introduction of this
problem in the summer course.


\end{document}